\documentclass[12pt]{amsart}

\usepackage{hyperref}
\usepackage{amssymb}

\usepackage{graphicx}
\usepackage{ifthen}
\usepackage{pict2e}
\usepackage{xargs}
\usepackage{xspace}
\usepackage{xcolor}
\usepackage{pgf,tikz}
\usepackage{pgfplots}
\usepackage{environ}
\usepackage{caption}
\usetikzlibrary{arrows}
\usepackage{enumitem}
\usepackage{xifthen}
\newtheorem{theorem}{Theorem}[section]
\newtheorem{lemma}[theorem]{Lemma}

\newtheorem{fact}[theorem]{Fact}

\theoremstyle{definition}\newtheorem{definition}[theorem]{Definition}

\theoremstyle{definition}
\theoremstyle{definition}

\numberwithin{equation}{section}

\newcommand{\bd}{\begin{definition}}
	\newcommand{\ed}{\end{definition}}

\DeclareMathOperator{\graph}{graph}
\DeclareMathOperator{\dom}{dom}
\DeclareMathOperator{\proj}{proj}

\DeclareMathOperator{\bij}{bij}

\newcommand{\concatt}{%
	\mathbin{\raisebox{1ex}{\scalebox{.7}{$\frown$}}}%
}

\newcommand{\bs}{\mathbf{\Sigma}^1_1}

\newcommand{\bdelta}[2]{\mathbf{\Delta}^{#1}_{#2}}
\newcommand{\bSigma}[2]{\mathbf{\Sigma}^{#1}_{#2}}
\newcommand{\bPi}[2]{\mathbf{\Pi}^{#1}_{#2}}

\newcommand{\Bairespace}{\omega^\omega}
\newcommand{\nDom}{\mathcal{ND}}
\newcommand{\HyperCompact}[1]{\mathcal{K}(#1)}
\newcommand{\HyperClosed}[1]{\mathcal{F}(#1)}
\newcommand{\Ztoinf}{\mathbb{Z}^\omega}
\newcommand{\Project}[1]{\proj_{#1}}
\newcommand{\PlusProd}{+^{p}}
\newcommand{\Bijection}{\bij}
\newcommand{\N}{\omega}
\newcommand{\lh}{lh}

\pgfplotsset{soldot/.style={color=blue,only marks,mark=*}}
\usetikzlibrary{math}
\usetikzlibrary{arrows.meta}
\usetikzlibrary{shapes,decorations}

\usetikzlibrary{decorations.pathmorphing}
\usepackage{pgfplots}

\title{Haar-positive closed subsets of Haar-positive analytic sets}

\author{M\'arton Elekes}
\address{Alfr\'ed R\'enyi Institute of Mathematics, Hungarian Academy of Sciences,
	PO Box 127, 1364 Budapest, Hungary and E\"otv\"os Lor\'and
	University, Institute of Mathematics, P\'azm\'any P\'eter s. 1/c,
	1117 Budapest, Hungary}
\email{elekes.marton@renyi.mta.hu}
\urladdr{http://www.renyi.hu/$\sim$emarci}

\author{M\'ark Po\'or}
\address{E\"otv\"os Lor\'and University, Institute of Mathematics, P\'azm\'any P\'eter s. 1/c, 1117 Budapest, Hungary}
\email{sokmark@gmail.com}

\author{Zolt\'an Vidny\'anszky}
\address{Kurt G\"odel Research Center for Mathematical Logic, Universit\"at Wien, W\"ah\-rin\-ger Strasse 25, 1090 Wien, Austria}
\email{zoltan.vidnyanszky@univie.ac.at}

\thanks{All three authors were supported by the National Research, Development and Innovation Office -- NKFIH, grants no.~124749, 129211. The first and third authors were also supported by the National Research, Development and Innovation Office -- NKFIH, grant no.~113047. The third author was also supported by FWF Grants P29999, P28153, and M2779.}

\begin{document}
	\maketitle
	\begin{abstract} We show that every non-Haar-null analytic subset of $\mathbb{Z}^\omega$ contains a non-Haar-null closed subset. Moreover, we also prove that the codes of Haar-null analytic subsets, and, consequently, closed Haar-null  sets in the Effros Borel space of $\mathbb{Z}^\omega$ form a $\mathbf{\Delta}^1_2$ set.
	\end{abstract}	
	
	It is not hard to see that non-locally-compact Polish groups do not admit a Haar measure (that is, an invariant $\sigma$-finite Borel measure). However, Christensen \cite{chr} (and later, independently, Hunt-Sauer-Yorke \cite{hsy}) generalized the ideal of Haar measure zero sets to every Polish group as follows:

	\begin{definition} Let $(G,\cdot)$ be a Polish group and $S \subset G$. 
		We say that $S$ is \emph{Haar-null}, (in symbols, $S \in \mathcal{HN}$) if there exists a universally measurable set $U \supset S$  (that is, a set measurable with respect to every Borel probability measure)
		and a Borel probability measure $\mu$ on $G$ such that for every $g,h  \in G$
		we have $\mu(gUh)=0$. Such a measure $\mu$ is called a \emph{witness measure} for $S$.
	\end{definition}
	
	This notion has found wide application in diverse areas such as functional analysis, dynamical systems, group theory, 
	geometric measure theory, and analysis (see, e.g., \cite{rosendal,banakhhaar,soleckihaar,soleckihaar2,cohen,aubry}). 
	It provides a well-behaved notion of ``almost every'' (or ``prevalent'') element of a Polish group. It is natural to investigate the regularity properties of Haar-null sets. In particular, one might wonder whether ``small sets are contained in nice small sets'' and whether ``large sets contain nice large sets''. Concerning the first question, Solecki \cite{soleckihaar2} has shown a positive statement, namely, that every analytic Haar-null set is contained in a Borel Haar-null set. On the negative side, the first and the third author \cite{haarnull} proved that, unlike the situation in locally-compact groups, in non-locally compact abelian Polish groups there are Borel Haar-null sets that have no $G_\delta$ Haar-null supersets. 
	
	In this paper we address the second question, and answer it positively in the case of a concrete non-locally compact Polish group, $\mathbb{Z}^\omega$, that is, the $\omega$'th power of the additive group of the integers:
	
	\begin{theorem} 
		\label{t:mainintro}
		Every analytic non-Haar-null subset of $\mathbb{Z}^\omega$ contains a closed non-Haar-null subset. 
	\end{theorem}	
	
	Our proof is based on the results of Solecki \cite{soleckihaar2} and Brendle-Hjorth-Spinas \cite{brendle}. Roughly speaking, a theorem from the former paper allows us to use witness measures of a very special form, and thus to reduce the understanding of the Haar-null ideal to the understanding of the non-dominating ideal (see Section \ref{s:prel} for the definitions), while the latter contains the regularity properties of the latter ideal. The reduction is based on a coding map and utilizes a compactness argument.
	
	We also calculate the exact projective class of the codes of the analytic Haar-null subsets of $\Ztoinf$, as well as the set $\{C \in \HyperClosed{\Ztoinf}:C \in \mathcal{HN}\}$, which turn out to be $\bdelta{1}{2}$.
	
	\section{Preliminaries and basic facts}
	\label{s:prel}
	
	We start with the most important definitions and theorems that will be used in the proof. We will adapt the notation from \cite{kechrisbook} for descriptive set theoretic concepts.

	The following fact is just a trivial consequence of standard results. 
	
	\begin{fact}
		\label{f:unif} Assume that $X,Y$ are Polish spaces, $F \subset X \times Y$ is Borel, $\mu$ is a Borel probability measure on $X$, and $K_0 \subset \Project{X}(F)$ is a compact set with $\mu(K_0)>0$. Then there exists a compact set $K \subset F$ such that $\Project{X}(K) \subset K_0$ and $\mu(\Project{X}(K))>0$. 
	\end{fact}
	
	\begin{proof}
		Using the Jankov, von-Neumann Uniformization theorem (see, \cite[Theorem 18.1]{kechrisbook}) there exists a measurable function $h:K_0 \to Y$ with $\graph(h) \subset F$. Consequently, by Lusin's theorem, there exists a compact set $K_1 \subset K_0$ with $\mu(K_1)>0$ and such that $h \restriction K_1$ is continuous. But then $K=\graph(h \restriction K_1)$ satisfies the required properties. 
	\end{proof}
	
	If $b \in \omega^{\omega}$ let us denote by $\mu_b$ the natural product probability measure on $\prod_{n \in \omega} [0,b(n)] \subset \Ztoinf$. For a Polish space $X$ we will denote by $\HyperCompact{X}$ and $\HyperClosed{X}$ the space of compact subsets of $X$ with the Hausdorff metric and the space of closed subsets of $X$ with the Effros Borel structure, respectively.   
	The following, easy to prove statements will be used:
	
	\begin{fact}
		\label{f:compact}
		Let $X$ be a Polish space and $F \subset X$ be closed. 
		Then 
		\begin{enumerate}
			\item the map $\Bairespace \to \mathcal{P}(\Ztoinf)$ (that is, the Polish space of the Borel probability measures on $\Ztoinf$) defined by $b \mapsto \mu_b$
			\item the map $\Bairespace \times \mathcal{K}(\Ztoinf) \to \mathbb{R}$ defined by $(b,K) \mapsto \mu_b(K)$
			\item the map $ \HyperCompact{X \times \Bairespace} \to \HyperCompact{X}$ defined by $K\mapsto \Project{X}(K)$
			\item \label{p:last} the set $\{K \in \HyperCompact{X}:K \subset F\}$
		\end{enumerate}
		is Borel. 
	\end{fact}
	
	For $f, g \in \Bairespace$ we will write $f \leq^* g$ if $f(n) \leq g(n)$ holds for each $n \in \N$ with finitely many exceptions. Recall that a set $S \subset \Bairespace$ is \emph{dominating} if it is cofinal in $\leq^*$. The $\sigma$-ideal of non-dominating sets is denoted by $\mathcal{ND}$.

	The following fact, essentially proved in \cite{brendle}, will play a crucial role. 
	
	\begin{fact}
		\label{f:contrest} Assume that $A \in \bs(\Bairespace)$ is a dominating set and $f:A \to \Bairespace$ is a Borel function. Then there exists a closed set $C \subset A$ such that $C \not \in \mathcal{ND}$ and $f\restriction C$ is continuous. 
	\end{fact}

	\begin{proof}[Sketch of the proof.]
	 Let $W=\{w_\sigma,s_\sigma:\sigma \in \omega^{<\omega}\}$ be a collection with the following properties: $w_\sigma \subset \omega$ and $\dom(s_\emptyset)\subset \omega$ are finite, $s_\emptyset:\dom(s_\emptyset) \to \omega$, and for $\sigma \not = \emptyset $ we have $s_\sigma: w_{\sigma\restriction \lh(\sigma)-1} \to \omega$ and for all $i \in w_{\sigma\restriction \lh(\sigma)-1}$ $s_\sigma(i)>\sigma(\lh(\sigma)-1)$, finally, for every $x \in \Bairespace$ we have that $\omega=\dom(s_\emptyset)\cup \bigcup_n w_{x \restriction n}$, where the union is disjoint. If $T \subset \omega^{<\omega}$ is a tree, define $C_{T,W}=$\[\{y\in \Bairespace:s_\emptyset \subset y \text{ and }\exists x \in [T] \ \forall n \in \omega (y\restriction w_{x \restriction n}= s_{x \restriction n+1}) \}.\]
	 
	 It is easy to see that the map $\phi_{T,W}:[T] \to C_{T,W}$ defined by assigning to $x \in [T]$ the unique $y \in C_{T,W}$ with the property that $\forall n \in \omega (y\restriction w_{x \restriction n}= s_{x \restriction n+1})$ is continuous and open. Moreover, one can also check that the set $C_{T,W}$ is closed for every $T$ and $W$.
	 
	 A \emph{Laver-tree} is a subtree $T$ of $\omega^{<\omega}$ such that it has a stem $s$ (that is, a maximal node with $\forall t \in T \ (t \subset s \lor s \subset t)$), and for every $t \not \subset s$ from $T$ the set $\{n\in \omega: t \concatt (n) \in T\}$ is infinite.

		In \cite{brendle} it is shown that $A$ contains a set of the form $C_{\omega^{<\omega},W}$. Consider now the Borel map $f \circ \phi_{\omega^{<\omega},W}:\Bairespace \to \Bairespace$. By \cite[Example 3.7]{hrusak} there exists a Laver-tree $T \subset \omega^{<\omega}$ such that 
		$f \circ \phi_{\omega^{<\omega},W} \restriction [T]$ is continuous. Clearly, $\phi_{T,W}=\phi_{\omega^{<\omega},W} \restriction [T]$, and since this map is open, $f\restriction \phi_{T,W}([T])(=C_{T,W})$ is continuous. Thus, it is enough to check that the set $C_{T,W}$ is dominating. 
		
		Let $z \in \omega^\omega$ be arbitrary. It is not hard to define an $x \in T$ such that $\phi_{T,W}(x) \geq^* z$ inductively: indeed, for every large enough $n$ (namely, if $n>\lh(s)$, where $s$ is the stem of $T$), if $x \restriction n$ is defined, we can find an $m$ so that $m>\max\{z(i):i \in w_{x \restriction n}\}$ and $x \restriction n \concatt (m) \in T$. Then for an $x$ obtained this way, it follows from $\forall i \in  w_{x \restriction n} \ (m<s_{x \restriction n+1}(i))$ that whenever $n$ is large enough and $i \in w_{x \restriction n}$ then $\phi_{T,W}(x)(i)>z(i)$. The fact that $|w_\sigma|<\aleph_0$ implies that $\phi_{T,W}(x)$ dominates $z$.  \end{proof}
	
	We will also make use of another consequence of the results in \cite{brendle}:
	
	\begin{fact}
		\label{f:complexityof} Let $A \subset \Bairespace \times \Bairespace$ be a $\bSigma{1}{1}$ set. Then the set $S_{\nDom}=\{x: A_x \in \nDom\}$ is $\mathbf{\Delta}^1_2$.
	\end{fact}
	\begin{proof}
		Counting the quantifiers in the definition of non-dominating sets gives that the set $S_{\nDom}$ is $\bSigma{1}{2}$. Now by \cite{brendle}, $x \in S_{\nDom}$ holds if and only if $\forall C \in \mathcal{F}(\Bairespace) \ (C \in \nDom \lor C \not \subset A_x)$. As it has been noted by Solecki, it follows from the construction in \cite{brendle} that the set $\{C\in \mathcal{F}(\Bairespace): C \in \nDom \}$ is $\bdelta{1}{2}$, thus, a straightforward calculation yields that $S_{\nDom}$ is $\bPi{1}{2}$ as well. 
	\end{proof}
	
	A connection between $\nDom$ and $\mathcal{HN}$ has already been established by Solecki \cite{soleckihaar2}. We will use the following:

	\begin{lemma}
		\label{l:coneof}
		Let $S \subset \Ztoinf$ be a Haar-null set. There exists a $b \in \Bairespace$ such that for every $b' \geq^* b$ we have that $\mu_{b'}$ is a witness for $S \in \mathcal{HN}$. 
	\end{lemma}
	Let us remark first that this statement has been implicitly proved in \cite{soleckihaar2} and used without proof in \cite{banakhhaar}. We will indicate how to show it using a slightly different argument from \cite{donat}.
	\begin{proof}[Sketch of the proof.]
		It is not hard to see that in order to establish the lemma it suffices to produce a $b$ such that for every $b' \geq b$ the measure $\mu_{b'}$ is a witness for $S \in \mathcal{HN}$. Now, one can check that in the proof of \cite[Theorem 3.1]{donat} only lower bounds are imposed on the sequence $(N(n))_{n \in \N}$ and consequently on the sequence $(a(n))_{n \in \N}$ as well. Applying this observation and \cite[Theorem 3.1]{donat} for a Borel Haar-null $B \supset S$, the choice $b=(a(n))_{n \in \N}$ yields the lemma.
	\end{proof}

	In this paper solely the group $\Ztoinf$ will be considered, and so we will use the additive notation for the group operation. 
	
	For an integer-valued function $f:X \to \mathbb{Z}$ the notation $|f|$ and $cf$ will be used for the function defined by $x \mapsto |f(x)|$ and $x \mapsto cf(x)$ for $x \in X$ and $c \in \mathbb{Z}$. Also we will write $f \leq g$ if for each $x \in X$ we have $f(x)\leq g(x)$. 
	
	\section{Complexity estimation}
    As a warm up, we calculate the complexity of the codes of Haar-null analytic subsets of  and the complexity of the closed Haar null subsets of $\Ztoinf$ in the Effros Borel space. It has been shown by Solecki \cite{soleckihaar2}, see also \cite{matheron,toden}, that the codes for the closed Haar-null subsets, as well as the set $\{C \in \HyperClosed{\Ztoinf}:C \in \mathcal{HN}\}$ are neither analytic nor co-analytic. (In fact, $\Ztoinf$ can be replaced by any non-locally compact Polish group admitting a two-sided invariant metric). So the next result is sharp.
	\begin{theorem}
		\label{t:complexity}
		The codes for Haar-null analytic subsets of $\Ztoinf$ form a $\bdelta{1}{2}$ subset, which is neither analytic nor co-analytic. More precisely, if $U \subset \Bairespace \times \Ztoinf$ is a $\bs$ set then the set $S=\{x \in \Bairespace:U_x \in \mathcal{HN}\}$ is $\mathbf{\Delta}^1_2 $ and there are closed sets $U$, for which this set is neither analytic nor co-analytic.  Moreover, the set $\{C \in \HyperClosed{\Ztoinf}:C \in \mathcal{HN}\}$ is also $\mathbf{\Delta}^1_2$.
	\end{theorem}
	\begin{proof}
		As mentioned above, by Solecki's results it is enough to prove that $S$ is $\bdelta{1}{2}$.
		By Fact \ref{f:complexityof} it suffices to define a $\bSigma{1}{1}$ subset $A$ of $\Bairespace \times \Bairespace$ with the property that $x \in S \iff A_x \in \nDom.$ Let \[(x,h) \in A \iff \exists g\in \Ztoinf (\mu_h(U_x+g)>0).\]
		It follows from \cite[Theorem 29.27]{kechrisbook} and Fact \ref{f:compact} that the set $A$ is $\bSigma{1}{1}$. 
		
		Let $x \in \Bairespace$ be arbitrary and assume that $U_x \not \in \mathcal{HN}$. We show that in this case $A_x=\Bairespace$. Indeed, for any $h \in \Bairespace$ the condition $U_x \not \in \mathcal{HN}$ implies the existence of a $g \in \Ztoinf$ with $\mu_{h}(U_x+g)>0$. 
		
		Now assume that $U_x \in \mathcal{HN}$, and towards a contradiction suppose that $A_x \not \in \nDom$. Then, we apply Lemma \ref{l:coneof} to $U_x$ and get a $b \in \Bairespace$. Pick an $h \in A_x$ such that $h \geq^* b$. Then on the one hand $\mu_h$ should witness that $U_x$ is Haar-null, on the other hand $\mu_h(U_x+g)>0$ for some $g \in \Ztoinf$, a contradiction.
		
		To see that the above argument implies that the set $\{C \in \HyperClosed{\Ztoinf}:C \in \mathcal{HN}\}$ is $\mathbf{\Delta}^1_2$, just fix a Borel isomorphism $\iota$ between $\Bairespace$ and $\HyperClosed{\Ztoinf}$. It is straightforward to check that the set $\{(C,h) \in \HyperClosed{\Ztoinf} \times \Ztoinf:h \in C\}$  is Borel and so is the set $B=\{(\iota(C),h) \in \HyperClosed{\Ztoinf} \times \Ztoinf:h \in C\}$. Now, using the first part of the statement, the set $\{x \in \Bairespace:\iota^{-1}(x)\in \mathcal{HN}\}=\{x \in \Bairespace:B_x\in \mathcal{HN}\}$ is $\mathbf{\Delta}^1_2$, and, consequently, its pullback under $\iota$, that is, the set $\{C \in \HyperClosed{\Ztoinf}:C \in \mathcal{HN}\}$ is $\mathbf{\Delta}^1_2$ as well.
	\end{proof}

	\section{The main result}

	In this section we prove Theorem \ref{t:mainintro}. Let us start with an easy observation. 
	\begin{lemma}
		Let $f \in \Bairespace$ be arbitrary. Then the set \[H(f)=\{g \in \Ztoinf: \exists^{\infty} n \in \omega \ |g(n)| \leq f(n) \}\] 
		is Haar-null in $\Ztoinf$.
		\label{l:nondom}
	\end{lemma}
	\begin{proof}
		Let $f'(n)=2^n (f(n)+1)$. We will show that the measure $\mu_{f'}$ witnesses that $H(f)$ is Haar-null. Let $h \in \Ztoinf$ be arbitrary. Clearly, \[H(f)+h=\bigcap_{k\in \N} \bigcup_{n \geq k} \{g+h:|g(n)| \leq f(n)\}, \]
		
		and for every $n\in \N$ we have \[\mu_{f'}(\{g+h:|g(n)| \leq f(n)\}) \leq \frac{2f(n)}{f'(n)}\leq \frac{2}{2^n},\]
		Thus, using $\sum_{n \in \N} \frac{2}{2^n}<\infty$ and the Borel-Cantelli lemma, we get that $\mu_{f'}(H(f)+h)=0$.   
	\end{proof}
	
	Now we are ready for the proof of the main result. Our strategy will be somewhat similar to the idea of the proof of Theorem \ref{t:complexity}, just significantly more sophisticated. To a given analytic set $A \not \in \mathcal{HN}$ we will assign a Borel set $D$ that encodes the witnesses for $A\not \in \mathcal{HN}$, i.e., codes for possible witness measures $\mu$ and compact sets $K$, and translations $t \in \Ztoinf$ with $K+t \subset A$ and $\mu(K)>0$. The coding will be constructed so that it ensures that $D$ is dominating. Using the results of Brendle, Hjorth, and Spinas, we will chose a dominating closed subset of $D$ with some additional properties, and from it a non-Haar-null subset of $A$ will be reconstructed. A compactness argument will yield that this set is in fact closed. 
	
	\begin{proof}[Proof of Theorem \ref{t:mainintro}]
		Let $A \in \mathbf{\Sigma}^1_1(\mathbb{Z}^\omega)$ be a non-Haar null set and let $F \in \HyperClosed{ \Ztoinf \times \Bairespace}$ with $\Project{\Ztoinf}(F)=A$. Fix a Borel bijection $\psi:2^\omega \to \HyperCompact{\Ztoinf \times \Bairespace}$ and define a Borel partial mapping $\phi:\Bairespace \times \Ztoinf \times \Ztoinf \to \HyperCompact{\Ztoinf \times \Bairespace}$ as follows: let $(b,t,c) \in \dom(\phi)$ iff the conjunction of the following holds:
		\begin{enumerate}
			\item \label{c:code} $c-t \in 2^\omega$.
			\item \label{c:bound} $2b \leq^* |t|$.
			\item \label{c:translate} $\Project{\Ztoinf}(\psi(c-t)) \subset \prod_{n \in \omega}[0,b(n)]$.
			\item \label{c:measure} $\mu_b(\Project{\Ztoinf}(\psi(c-t)))>0$.
		\end{enumerate}
		
		Let us use the notation $\PlusProd$ for the $(\Ztoinf \times \Bairespace) \times \Ztoinf \to \Ztoinf \times \Bairespace$ mapping that is the translation of the first coordinate, i.e., $(r,x) \PlusProd t=(r+t,x)$. 
		Define $\phi$ for $(b,t,c) \in \dom(\phi)$ by letting  \[\phi(b,t,c)=\psi(c-t)\PlusProd t,\]
		in other words, $\phi(b,t,c)=\psi(c-t)\PlusProd t$ is the compact subset of $\Ztoinf \times \Bairespace$ defined by $\{(r,x)\PlusProd t: (r,x) \in \psi(c-t)\}$.
		
		Finally, we will need a homeomorphism $\Bijection:\Bairespace \to \Bairespace \times \Ztoinf \times \Ztoinf$. In order to be precise, let us fix a concrete one by letting for every $n \in \N$
		\begin{itemize}
			\item $\Bijection(f)(0)(n)=f(3n),$
			\item for $i \in \{1,2\}$ define $\Bijection(f)(i)(n)=$ \[
			\begin{cases}
			f(3n+i)/2, \text { if $f(3n+i)$ is even,}\\
			-(f(3n+i)+1)/2,\text { if $f(3n+i)$ is odd.}
			\end{cases}
			\]
		\end{itemize}
		
		\begin{lemma}
			\label{l:dom}
			Let $D=\{f \in \Bairespace: \phi(\Bijection(f)) \subset F\}$. Then $D$ is dominating.
		\end{lemma}
		\begin{proof}
			
			Assume otherwise, and let $f \in \Bairespace$ witness this fact. 
			Without loss of generality, we can assume that $f$ is constant on the sets of the form $\{3n,3n+1,3n+2\}$, and it attains only positive values. Define an element $f' \in \Ztoinf$  by $f'(n)=2(f(3n)+1)$. 
			
			Using Lemma \ref{l:nondom} and $A \not \in \mathcal{HN}$ we get that $A \setminus H(3f') \not \in \mathcal{HN}$. This yields that there exist a compact set $K_0 \subset \prod_n[0,f'(n)]$ with $\mu_{f'}(K_0)>0$ and a $t \in \Ztoinf$ such that $K_0+t \subset A \setminus H(3f')$. Now, using Fact \ref{f:unif} for the sets $F\PlusProd(-t)$, $K_0$ and $\mu_{f'}$ we get a compact set $K \subset F\PlusProd(-t)$ (or, equivalently, $K\PlusProd t \subset F$) with $\mu_{f'}(\Project{\Ztoinf}(K))>0$ and $\Project{\Ztoinf}(K) \subset K_0$. 
			
			Consider now the function  $g=\Bijection^{-1}(f',t,t+\psi^{-1}(K)).$
			We claim that $g \geq^* f$ and $g \in D$, contradicting our initial assumption and thus finishing the proof.
			In order to see $g \geq^*f$, notice that, as $\emptyset \neq K_0+t \subset(\prod_n[0,f'(n)]+t) 
			\cap (A \setminus H(3f'))$, necessarily $f'+|t| \geq^* 3  f'$, so $|t| \geq^* 2f'$. Then, it is straightforward to check from the definition of $\Bijection$ that $g \geq^* f$ holds. 
			
			Checking $g \in D$ is just tracing back the definitions: clearly, $\Bijection(g)=(f',t,t+\psi^{-1}(K)) \in \dom(\phi)$ holds, as we have already seen \eqref{c:code}, \eqref{c:bound}, and for \eqref{c:translate}, \eqref{c:measure} note that $\psi(t+\psi^{-1}(K)-t)=K$ and $\Project{\Ztoinf}(K) \subset K_0$. Finally, $\phi((f',t,t+\psi^{-1}(K)))=K \PlusProd t\subset F$. 
		\end{proof}
		Using Fact \ref{f:compact} we get that $\dom(\phi)$ is a Borel set, and as $\PlusProd$ is continuous, $\phi$ is a Borel map. Moreover, using \eqref{p:last} from Fact \ref{f:compact}, $D$ must be Borel as well. Since non-dominating sets form a $\sigma$-ideal, by passing to a dominating Borel subset of $D$, we can assume that there is an $n_0 \in \omega$ and a sequence $(\beta, \tau, \gamma) \in \omega^{n_0} \times \mathbb{Z}^{n_0} \times \mathbb{Z}^{n_0}$ such that for each $f \in D$ if $\Bijection(f)=(b,t,c)$ then $2b(k) \leq |t(k)|$ for each $k \geq n_0$ and for each $k<n_0$ we have $\beta(k)=b(k), \tau(k)=t(k), \gamma(k)=c(k)$. 
		
		Now Fact \ref{f:contrest} implies the existence of a closed dominating set $C \subset D$ such that $\phi \circ \Bijection \restriction C$ is continuous. We claim that the set $C'=\Project{\Ztoinf}(\bigcup_{x \in C} \phi(\Bijection(x)))$ is closed and non-Haar null, which finishes the proof, as it is clearly a subset of $A$. 
		
		First, we show that the set is closed. Let $r_n \in C'$ with $r_n \to r$ and assume that $r_n \in \Project{\Ztoinf}(\phi(b_n,t_n,c_n))$, where $\Bijection^{-1}(b_n,t_n,c_n) \in C$. Then, $r_n\in \Project{\Ztoinf}(\phi(b_n,t_n,c_n))=\Project{\Ztoinf}(\psi(c_n-t_n)\PlusProd t_n)$ and by $(b_n,t_n,c_n) \in \dom(\phi)$ we get that $\Project{\Ztoinf}(\psi(c_n-t_n)) \subset \prod_k[0,b_n(k)]$ and by our assumptions on $D$ we have $2b_n(k) \leq |t_n(k)|$ for $k \geq n_0$. Then $|r_n(k)| \geq |t_n(k)|-|b_n(k)| \geq |t_n(k)|/2$ and so $2|r_n(k)|+1 \geq \max\{|b_n(k)|,|t_n(k)|,|c_n(k)|\}$. By our assumptions on the $n_0$-long initial segments of the elements of $\Bijection(D)$, and the convergence of $r_n$ we get that the sequence $(b_n,t_n,c_n)_{n \in \omega}$ must contain a convergent subsequence, and, as $\Bijection$ is a homeomorphism, $\Bijection^{-1}(b_n,t_n,c_n)$ contains such a subsequence as well. If $f \in \Bairespace$ is its limit then of course $f \in C$, and the continuity of $\Project{\Ztoinf} \circ \phi \circ \Bijection \restriction C$ yields that $r \in \Project{\Ztoinf} (\phi (\Bijection(f))) \subset C'$ holds.

		Second, assume that $C'$ is Haar-null. By Lemma \ref{l:coneof} there exists a $b \in \Bairespace$ such that for each $b' \geq^*b$ the measure $\mu_{b'}$ witnesses $C' \in \mathcal{HN}$. Since the set $C$ is dominating, there exists an $f \in C$ such that $f(3n) \geq b(n)$ holds for every large enough $n$. Then if $\Bijection(f)=(b',t',c')$, by definition $b'(n)=f(3n)$, so $\mu_{b'}$ must witness that $C'$ is Haar-null. On the other hand, $\Project{\Ztoinf}(\phi(b',t',c'))=\Project{\Ztoinf}(\psi(c'-t') \PlusProd t')=\Project{\Ztoinf}(\psi(c'-t'))+t' \subset C'$, so $\mu_{b'}$ must witness that the set $\Project{\Ztoinf}(\psi(c'-t'))$ is Haar-null as well. But, $(b',t',c') \in \dom(\phi)$ holds, so by \eqref{c:measure} we have $\mu_{b'}(\Project{\Ztoinf}(\psi(c'-t')))>0$, a contradiction. \end{proof}

	\bibliographystyle{abbrv}
	\bibliography{bblanalytic}
	
\end{document}